\documentclass[11pt]{article}

\usepackage{pstricks}
\usepackage{amssymb}
\usepackage{epsfig}
\usepackage[mathscr]{eucal}
\usepackage{colortbl}
\usepackage{multirow}

\newcommand{\Natural}{\mathbb{N}}

\newcommand{\C}{\mathbb{C}}

\renewcommand{\epsilon}{\varepsilon}
\newcommand{\beq}{\begin{equation}}
\newcommand{\eeq}{\end{equation}}
\newcommand{\bea}{\begin{eqnarray}}
\newcommand{\eea}{\end{eqnarray}}
\newcommand{\bean}{\begin{eqnarray*}}
\newcommand{\eean}{\end{eqnarray*}}

\newtheorem{theorem}{Theorem}
\newtheorem{lemma}[theorem]{Lemma}

\newtheorem{proposition}[theorem]{Proposition}
\newtheorem{definition}[theorem]{Definition}

\newenvironment{proof}%
{\par\noindent{\em Proof.\ }}%
{\ \hfill ~\rule{2mm}{2mm}\par\medskip}
{\par\noindent{\bf Nota:\ }}%
{\ \hfill \par\medskip}
\newenvironment{remark}%
{\par\noindent{\bf Remark.\ }}%
{\ \hfill \par\medskip}

\def\qh{\mathcal{P}}

\def\qh{\mathscr{P}}
\def\QH{\mathcal{Q}}

\def\Opl{{\ell}}

\def\gx{g}

\def\zero{{\mathbf 0}}

\def\x{{\mathbf x}}
\def\X{{\mathbf X}}

\def\C{{\mathbf C}}
\def\D{{\mathbf D}}
\def\F{{\mathbf F}}
\def\G{{\mathbf G}}

\def\t{{\mathbf t}}

\def\co{{\mathrm{c}}}

\def\det{{\mathrm{det}}}

\def\dim{{\mathrm{dim}}}
\def\Cor{{\mathrm{Cor}}}
\def\Proy{{\mathrm{Proy}}}
\def\Ker{{\mathrm{Ker}}}
\def\Range{{\mathrm{Range}}}

\def\parent#1{\left( #1 \right)}

\def\llave#1{\left\{ #1 \right\}}

\def\fracp#1#2{{\textstyle{\frac{#1}{#2}}}}
\def\matriz#1{\parent{\begin{array} #1 \end{array} } }
\def\matriz22#1#2#3#4{\parent{\begin{array}{cc} #1&#2\\#3&#4 \end{array} } }

\def\0{{\bf 0}}
\begin{document}

\title{\bf The analytic integrability problem for perturbations of homogeneous quadratic Lotka-Volterra systems}

\author{Antonio Algaba, Crist\'obal  Garc\'{\i}a, Manuel Reyes \\
\small Dept.  Matem\'aticas, Facultad de Ciencias,  Univ. of Huelva, Spain. \\
\small {\rm e--mails:} {\tt algaba@uhu.es}, {\tt cristoba@uhu.es}, {\tt colume@uhu.es}\\
}
\date{}
\date{\today}

\maketitle

\abstract{We solve the analytic integrability problem for differential systems in the plane whose origin is an isolated singularity and the first homogeneous component is a quadratic Lotka-Volterra type. As an application, we give the analytically integrable systems of a class of systems $\dot{x}=x(P_1+P_2),\ \dot{y}=y(Q_1+Q_2),$ being $P_i,Q_i$ homogeneous polynomials  of degree $i$. } 

\section{Introduction and statement of the main result}
We focus on the study of the analytic integrability of a planar differential system
\bea  \label{sys1}
\dot{\x}= \F(\x)
\eea
where $\F$ is analytic  in a neighborhood of the origin.

Writing the Taylor expansion $\F=\F_n+\F_{n+1}+\cdots,\ \F_n\not\equiv 0,$ we notice that the condition of polynomial integrability of $\F_n,$ lowest degree homogeneous term of the vector field,   is a necessary condition in order to be $\F$ analytically integrable.

The analytically integrable differential systems with non-null linear part, i.e. $n=1,$ and $\F_1$ polynomially integrable, are orbitally linearizable. Indeed, it has the following cases in function the eigenvalues of $D(\F_1)(\0)$:
if $\lambda_1 \lambda_2 \ne 0,$ the origin is  either a saddle, or node or a non-degenerate monodromic  singular point (with complex eigenvalues). If $\lambda_1=0$ and $\lambda_2 \ne 0,$ the origin is a saddle-node.  Finally if $\lambda_1=\lambda_2=0,$ it has a nilpotent singular point.
The nodes and saddle-nodes are not analytically integrable. A non-degenerate monodromic point is analytically integrable if, and only if, it is orbitally equivalent to $(-y,x)^T,$ and a resonant saddle has an analytic first integral around the singular point if, and only if, it is orbitally equivalent to $(px,-qy)^T$ with $p,q\in \Natural,$ see \cite{HJ,Poincare2}.  The most studied systems whose origin is a resonant saddle are the Lotka-Volterra systems, see \cite{CGRS2012,CMR,CR,DGOR2013,GR,LCC,LCL,WH2014} and references therein.

Recently, it is proved that a nilpotent singular point is analytically integrable if, and only if, the vector field is orbital equivalent to its lowest degree quasi-homogeneous term, see \cite{AGG17}.


For vector fields with null linear part (a degenerate singular point) some partial results are known. 
Not any analytically integrable vector field with null linear part is orbitally equivalent to its first quasi-homogeneous component, see \cite{ACG0,ACG}.
The analytic integrability problem when the first quasi-homogeneous component of $\F$ is conservative whose Hamiltonian function $h$ has only simple factors is completely solved in \cite{AlgabaNonlinearity09}. In \cite{Algaba_Checa_Garcia_Gine2015} it is studied a particular case with $h$ having multiple factors.
\smallskip

In this work, we deal with perturbations of homogeneous quadratic Lotka-Volterra systems, $$\F=\F_2+\cdots,\ \ \F_2(x,y)=(xP_1(x,y),yQ_1(x,y))^T$$ with $P_1$ and $Q_1$ homogeneous polynomials of degree one (vector field with null linear part) and the origin is an isolated singular point of $\dot{\x}=\F_2(\x).$    

Here, we solve the analytic integrability problem for these systems. More specifically, we prove that, under the condition of polynomial integrability of $\F_2,$ the vector field $\F$ is analytically integrable if, and only if, it is orbitally equivalent to its lowest degree component. (Theorem \ref{main}).

As consequence, we characterize its analytic integrability through the existence of a Lie symmetry (Theorem \ref{main2}) and of an inverse integrating factor (Theorem \ref{main3}).

We emphasize that for the vector fields $\F=\F_3+\cdots, $ whose first homogeneous component is $\F_3(x,y)=(xP_2(x,y),yQ_2(x,y))^T,$ with $P_2$ and $Q_2$ homogeneous polynomials of degree two, the existence of an analytic first integral is not equivalent to the orbital equivalence of its lowest degree component. In fact, the vector field $(x(-3y^2-x^2),y(y^2+3x^2))^T+(y^4,0)^T$ is a perturbation of a cubic Lotka-Volterra type, analytically integrable since it is Hamiltonian. Nevertheless, it is not possible to transforms it into its lowest degree component. So, for $n\ge 3$, the problem is still open.

Finally, in Section \ref{sec:aplica}, we calculate the systems 
$$\left(\begin{array}{c}\dot{x}\\ \dot{y}\end{array}\right)=\left(\begin{array}{c}x(-x+3y)\\ y(3x-y)\end{array}\right)+\left(\begin{array}{c}x(a_{20}x^2+a_{11}xy+a_{02}y^2)\\ y(b_{20}x^2+b_{11}xy+b_{02}y^2)\end{array}\right)$$ with an analytic first integral at the origin.

\subsection{Invariant curves and first integrals of vector fields}

\smallskip

First we give the definition of invariant curve and its associated cofactor.

We deal with a vector field $\F=(P,Q)^T$ with $P,Q$ analytic at the origin and $P(\0)=Q(\0)=0.$ Throughout the paper, we will denote by $F$ the operator associated to the vector field $\F$, that is, $F:=P\partial_x+Q\partial_y.$  
We recall the concept of invariant curve and its associated cofactor.
\begin{definition} It is said that $C\in\C[[x,y]]$ (algebra of formal power series in $x,y$ over $\C$), with $C(\0)=0,$ is an invariant curve of the vector field $\F,$
if there exists $K\in \C[[x,y]],$ named cofactor of $C$, such that $F(C)
=KC.$ \\
Moreover, if $K\equiv 0,$ it is said that $\F$ is formally integrable and $C$ is a first integral of $\F$.  
\end{definition}
Let note that any formal function $C$ with $C(\0)\ne 0,$ satisfies $F(C)=KC$ with $K=F(C)/C\in\C[[x,y]].$\\

We will denote by $\qh_k$ the vector space of homogeneous scalar polynomials of degree $k,$ and by $\QH_{k}$ the vector space of polynomial homogeneous vector fields of degree $k$. We will use Taylor expansion of functions and vector fields without to consider questions of convergence. 
We note that analytic integrability is equivalent to formal integrability, see  Mattei \& Moussu \cite{Mattei80}.

Throughout the paper, we will be denoted by $\D=(x,y)^T\in
\mathcal{Q}_{1}$ (dissipative homogeneous vector
field) and by $\X_{h}=(-\partial h/\partial y,\partial
h/\partial x)^T$ (Hamiltonian vector field associated to the
polynomial $h$). \\
The following splitting of a homogeneous vector field plays a main role in our study.
\begin{proposition}{\cite[Prop.2.7]{AlgabaNonlinearity09}}  
Every $\F_k\in\mathcal{Q}_k$ can be uniquely written as
$\F_k=\X_{h}+\mu\D$
  with $ h:=\frac{1}{k+1}(\D\wedge\F_k)\in\mathscr{P}_{k+1}$ (product wedge
of both vector fields) and $
\mu:=\frac{1}{k+1}\mbox{div}(\F_k)\in\mathscr{P}_{k-1}$ (divergence of
$\F_k$).
\end{proposition} 

We give the Taylor expansion of a formal invariant curve of a formal vector field. 
\begin{proposition}\label{pro:primer} Consider $\F=\sum_{j\geq n}\F_j,\ \F_j\in\QH_j$ with $\F_n\not\equiv \0.$ Let $C$ a formal invariant cuve of $\F$ with cofactor $K$. Then, $C=\sum_{j\geq s}C_j,\ C_j\in\qh_j$ and $K=\sum_{j\geq n}K_j$, $K_j\in\qh_j,$ being the polynomial $C_s$ an invariant curve of  the polynomial vector field $\F_n$ with cofactor $K_n.$
\end{proposition}
\begin{proof} 
It is enough to consider the lowest degree homogeneous term of the equality $F(C)-KC=0$.
\end{proof}

The following two results show the invariant curves of a homogeneous vector field.
\begin{proposition}\label{pro:ini1}
Every homogeneous polynomial invariant curve of a homogeneous vector field $\F_n$ is given by $g_1^{n_1}g_2^{n_2}\dots g_m^{n_m}$ being each $g_j$ a polynomial invariant curve of $\F_n.$\\
Moreover, its cofactor is $n_1K_1+\cdots+n_mK_m,$ being $K_j$ the cofactor of $g_j.$ 
\end{proposition}
\begin{proof}
We suppose that $g=g_1p,$ ($g_1$ an irreducible homogeneous polynomial), is an invariant curve of $\F_n$ with $K_n$ cofactor of $g.$ It has that   $F_n(g_1p)=g_1F_n(p)+pF_n(g_1)=K_ng_1p,$ that is,
$g_1(pK_n-F_n(p))=pF_n(g_1).$ 
From the irreducibility of $g_1$, it has two situations: either $g_1$ is an irreducible invariant curve of $\F_n$, in such case, $p$ is also an invariant curve of $\F_n$ and we repeat the process for $p$. Or $p=qg_1,$ i.e. $g=g_1^2q$. We now have that $F_n(g_1^2q)=g_1^2F_n(q)+2qg_1F_n(g_1)=K_ng_1^2q.$ Thus, 
$g_1(qK_r-F_r(q))=2qF_r(g_1).$ Reasoning of similar way, it completes the proof.\\
The second part is obtained easily.
\end{proof}

\begin{proposition}\label{pro:ini2} Given $\F_n\in\QH_n,$ any factor of $h$ 
is an invariant curve of $\F_n.$ Conversely, any homogeneous polynomial invariant curve of $\F_n$ is a factor of $h.$ \\
Moreover, if $I$ is a polynomial first integral of $\F_n,$ then $I=g_1^{n_1}g_2^{n_2}\cdots g_m^{n_m}$ where $g_1,\ \ldots, g_{n_m}$ are all the irreducible factors of $h$ and $n_i>0.$
\end{proposition}
\begin{proof} 
We know that $\F_n=\X_{h}+\mu\D$ with $\mu=\frac{1}{n+1}\mbox{div}(\F_n).$ 
Let $f\in\qh_{s}$ a factor of $h$ then $h=fg$ and $F_n(f)=X_{fg}(f)+\mu D(f)=fX_g(f)+s\mu f=(X_g(f)+s\mu)f$ Therefore, $f$ is an invariant curve of $\F_n.$\\
If $f\in\qh_s^\t$ is an irreducible invariant curve of $\F_n$ with cofactor $K_n$ then $K_nf=F_n(f)=X_{h}(f)+\mu D(f)=X_{h}(f)+s\mu f.$ Thus, $X_{h}(f)=(K_n-s\mu)f$ and $f$ is an invariant curve of $\X_{h}.$ So, $f$ divides to $h.$\\
Last on, if $I$ is a first integral of $\F_n$, it is an invariant curve of $\F_n,$ that is, from Proposition \ref{pro:ini1}, a factorization of $I$ is formed by the irreducible factors of $h$. On the other hand, a first integral is zero on every invariant curve. So, $n_i>0.$    
\end{proof}

\subsection{Necessary condition of analytic integrability}

Now we study the integrability problem for a vector field whose first homogeneous component is a quadratic type Lotka-Volterra.
The following result determines the expression of the lowest degree component in the case of polynomial integrability of this class of vector fields.
\begin{proposition}[Necessary condition of analytic integrability]\label{pronilint}
Let $\F=\F_2+\mbox{h.o.t.}$ be with $\F_2=(x(a_1x+a_2y),y(b_1x+b_2y))^T$ such that the origin of $\dot{\x}=\F_2(\x)$ is isolated ($a_1,\, b_2,\, a_2b_1-a_1b_2$ are different from zero).  If $\F$ is formally integrable,  then $b_1-a_1,\, b_2-a_2$ are different from zero, and  there exists $\Phi_1\in\QH_1$, $\det\parent{D\Phi_1(\zero)}\neq 0$ such that $\G:=\parent{\Phi_1}_*\F=\G_2+\mbox{h.o.t.}$, being $$\G_2=(x(-qx+(q+r)y),y((p+r)x-py))^T,\ p,q,r\in\Natural,\ \mbox{gcd}(p,q,r)=1$$  and $I_M=x^py^q(y-x)^r$ is a polynomial first integral of $\G_2$ of degree $M=p+q+r.$
\end{proposition}
\begin{proof} Let $I=I_N+\mbox{h.o.t.}$ be a formal first integral of $\F.$ Equation $F(I)=0$ for degree $N+1$ is $F_2(I_N)=0,$ i.e. $\F_2$ is polynomially integrable and $I_N$ is a first integral of $\F_2.$\\
Polynomial $h$ associated to $\F_2$ is $h=\frac{1}{3}xy((b_1-a_1)x+(b_2-a_2)y).$\\
We suppose that $a_1=b_1$ (analogously for $a_2=b_2$). 
From Proposition \ref{pro:ini2}, if there would exist a first integral of $\F_2$, it would have the expression $I_N=x^{n_1}y^{n_2}$ with $n_i>0.$ So, $F_2(I_N)=0$ becomes 
$$n_1a_1+n_2b_1=0,\qquad n_2b_2+n_1a_2=0.$$ Therefore, $n_1=n_2=0$ since  $a_2b_1-a_1b_2\ne 0.$ This contradicts the existence of the first integral.\\
We assume that $b_1\ne a_1$ and $b_2\ne a_2.$ From Proposition \ref{pro:ini2}, the first integral is $I_{Ms}=x^{ps}y^{qs}((b_1-a_1)x+(b_2-a_2)y)^{rs}$ with $p,q,r,s$ natural numbers, $ \mbox{gcd}(p,q,r)=1$ and $M=p+q+r.$ By imposing $F_2(I_{Ms})=0,$ it has that 
 $$(p+r)a_1+qb_1=0,\qquad (q+r)b_2+pa_2=0,$$ i.e. 
$a_2 = -\frac{b_2(q+r)}{p}$ and $b_1 = -\frac{a_1(p+r)}{q}.$\\ The linear change $\Phi_1(x,y,t)=(-pa_1x,-qb_2y,st/pq)$ transforms $\F_2$ into $\G_2= (x(-qx+(q+r)y),y((p+r)x-py))^T $
being $I_M=x^py^q(y-x)^r$ a first integral of $\G_2.$

\end{proof}

\section{Normal Form for perturbations of homogeneous quadratic Lotka-Volterra systems}

We will not consider questions of convergence in the normal forms because the formal integrability is equivalent to the analytical integrability for the vector fields analyzed, see \cite{Mattei80}.

In \cite{AGG17}, the authors provide an orbital normal form of the vector field whose first quasi-homogeneous term is non-conservative. Here, we provide the expression of the normal form for the vector field $\F=\F_2+\mbox{h.o.t.}$ with $\F_2=\X_h+\mu\D\in\QH_2.$

For every $k\in\mathbb{N}$, we fix the subspaces $\Delta_{k+2}$ such that $\qh_{k+2}=\Delta_{k+2}\bigoplus h \qh_{k-1}.$ We consider the linear operators:\\
\bean
\begin{array}{rcl}\Opl_{k}&:&\qh_{k-1}
\longrightarrow \qh_{k}^{\t}\\
 && \eta_{k-1} \longrightarrow F_2(\eta_{k-1}),\end{array}
\eean and 
\bean
 \begin{array}{rcl}
 \Opl_{k+3}^{\co} &:& \Delta_{k+2} \longrightarrow 
\Delta_{k+3} \\ &&
\gx_{k+2} \longrightarrow  \Proy_{\Delta_{k+3}}
(F_{2}-\fracp{3}{k+3}\mu D)(\gx_{k+2}).
 \end{array} 
 \eean

\begin{theorem}\label{teoFNlequivFn} Let $\F=\sum_{j\geq 2}\F_j$, $\F_j\in\QH_j$. If $\Ker\parent{\Opl_{k+3}^{\co}}=\llave{0}$ for all $k\in\mathbb{N}$ then $\F$ is orbitally equivalent to $$\G=\F_2+\sum_{j>2}\G_{j}, \ \mbox{with} \ \G_{j}=\X_{g_{j+1}}+\eta_{j-1}\D\in\QH_{j},$$
 where $g_{j+1}\in\Cor\parent{\Opl^{c}_{j+1}}$ and $\eta_{j-1}\in\Cor\parent{\Opl_{j-1}}$.\ (where $\mbox{Cor}(\cdot)$ is a complementary subspace to $\mbox{Range}(\cdot)).$
 \end{theorem}

Next results are referred to vector fields whose first homogeneous component is polynomially integrable and quadratic Lotka-Volterra type.

\begin{lemma}\label{lebasepkt}
Consider $\F_2= (x(-qx+(q+r)y),y((p+r)x-py))^T$ with $p,q,r$ natural numbers. It has that 
for all $k\in\mathbb{N},\ \Ker\parent{\Opl_{k+3}^{\co}}=\llave{0}$. Moreover, $\Cor\parent{\Opl_{k+3}^{\co}}=\llave{0}$.
\end{lemma}

\begin{proof}
Vector field $\F_2=\X_h+\mu\D$ with $h=\frac{p+q+r}{3}xy(x-y)$ and $\mu=\frac{1}{3}((-2q+p+r)x+(q+r-2p)y).$
We choose the bases $\Delta_{k+2}=\langle x^{k+2},\, x^{k+1}y,\, y^{k+2}\rangle$ and 
$\Delta_{k+3}=\langle x^{k+3},\, x^{k+2}y,\, y^{k+3}\rangle.$\\
Consider $$\G_2^{(k+3)}=\F_2-\fracp{3}{k+3}\mu\D=
\left(\begin{array}{l}
(-q-\fracp{-2q+p+r}{3+k})x^2+(q+r-\fracp{q+r-2p}{3+k})xy\\
(p+r-\fracp{-2q+p+r}{3+k})xy+(-p-\fracp{q+r-2p}{3+k})y^2
\end{array}\right).$$
We have that 
\bean
G_2^{(k+3)}(x^{k+2})&=&A_1x^{k+3}+B_1x^{k+2}y,\\
G_2^{(k+3)}(x^{k+1}y)&=&A_2x^{k+2}y+B_2x^{k+1}y^2=(A_2+B_2)x^{k+2}y-\fracp{3B_2}{p+q+r}x^kh,\\
G_2^{(k+3)}(y^{k+2})&=&A_3xy^{k+2}+B_3y^{k+3}=A_3x^{k+2}y+A_3hp_k(x,y)+B_3y^{k+3},
\eean
with
\bean
A_1&=&-\fracp{2+k}{3+k}(q+qk+p+r),\\
B_1&=&\fracp{2+k}{3+k}(2q+qk+2r+rk+2p),\\
A_2+B_2&=&\fracp{2+k}{3+k}(p+q+r+rk),\\
A_3&=&\fracp{2+k}{3+k}(2p+pk+2r+rk+2q),\\
B_3&=&-\fracp{2+k}{3+k}(p+pk+q+r),
\eean
and $p_k$ homogeneous polynomial of degree $k$.
In this way, the determinant of the matrix of the operator $\Opl_{k+3}^{\co}$ is 
$$ \fracp{(k+2)^3}{(k+3)^3}(q+qk+p+r)(q+p+rk+r)(p+pk+q+r),$$
which is different from zero. Therefore, both  $\Ker\parent{\Opl_{k+3}^{\co}}$ and $\Cor\parent{\Opl_{k+3}^{\co}}$ are trivial subspaces.
\end{proof}
For computing  $\Cor\parent{\Opl_k}$ with $k>n$, we need the following two technical lemmas.

\begin{lemma}\label{irred1}
Consider $\F_{n}\in \QH_{n}$
irreducible and $f \in \mathbb{C}[x,y]$ an irreducible invariant curve of  $\F_{n}$. If $F_n(p_k) \in \left<f\right>$ with $p_k\in \qh_{k},$
then $p_k \in \left<f\right>$.
\end{lemma}
\begin{proof}
If $F_n(p_k)= 0$ then $p_k$ is a first integral of
$\dot{\x}=\F_{n}$. A first integral of $\F_{n}$ vanishes on any invariant curve of it, i.e., $p_k(\x)=0$ when
$f(\x)=0$. Therefore, by Hilbert`s Nullstellensatz $p_k\in rad\left<f\right>$.
Since $\left<f\right>$ is a prime ideal, then  $\left<f\right>=rad\left<f\right>$, in consequence
$p \in \left<f\right>$.

If $F_n(p_k)\neq 0$, let
$\nu\in \mathbb{C}[x,y]\setminus\llave{0}$ such that
$f\nu=F_n(p_k)$. Consider $\gamma(t)$, real or
complex, a solution curve of $\dot{\x}=\F_{n}(\x)$ which is a
parametrization of $f(\x)=0$. We assume that $\lim_{t\to
-\infty}\gamma(t)=0$, (the other case $\lim_{t\to
+\infty}\gamma(t)=0$ is proved in a similar way). Taking into account that $p_k(\zero)=0$
then
\bean p_k(\gamma(t))&=&p_k(\gamma(t))-p_k(\zero)=\int_{-\infty}^t \fracp{d p_k(\gamma(s))ds}{ds}=\int_{-\infty}^t \nabla_\x p_k \cdot\F_n(\gamma(s))ds \\
&=&\int_{-\infty}^t
F_n(p_k)(\gamma(s))ds=\int_{-\infty}^t f (\gamma(s))\nu(\gamma(s))ds=0.\eean
 Recalling that $f(\x)=0$ is the union of orbits, we have that $p_k(\x)=0$ when $f(\x)=0$. Therefore, by Hilbert`s Nullstellensatz $p_k\in rad\left<f\right>$. Since $\left<f\right>$ is a prime ideal, then  $\left<f\right>=rad\left<f\right>$, in consequence
$p_k \in \left<f\right>$.
\end{proof}
\begin{remark} The hypothesis of the irreducibility of the invariant curve is fundamental. For instance, if we consider $\F_{2}:=(-2x^2,-3x^2-2xy+3y^2)^T\in\QH_2$ irreducible and the invariant curve $(y-x)^2,$ for $p_3=x^2(y-x)$ we have that $F_{2}(p_3)=3x^2(y-x)^2\in \langle (y-x)^2\rangle$ and nevertheless $p_3\notin\langle (y-x)^2 \rangle$.
\end{remark}
\begin{lemma}\label{irredfm} Consider $\F_2=(x(-qx+(q+r)y),y((p+r)x-py))^T$ with $p,q,r$ natural numbers. Let $k$ and $m$ natural numbers with $p+q+r\ne p\frac{k}{j},\ p+q+r\ne q\frac{k}{j},\  p+q+r\ne r\frac{k}{j},\ j=1,\dots, m-1.$ If $p_k\in \qh_{k}$ such that $F_2(p_k) \in \left<f_i^m\right>,$ being $f_1=x,\, f_2=y,\, f_3=x-y,$ invariant curves of $\F_2$, then $p_k \in \left<f_i^m\right>,\ i=1,2,3.$  \end{lemma}
\begin{proof} 
We prove the case $i=1,\ (f_1=x),$ the cases $i=2,3$ are analogous.\\
Lemma \ref{irred1} proves the statement for $m=1$. 

We first consider the case $m=2$. We denote by $K_1=-qx+(q+r)y$ the cofactor of $x$. If $F_{2}(p_k)\in\langle x^2\rangle$ then $F_{2}(p_k)\in\langle x\rangle$ and by Lemma \ref{irred1} we have that there exists $p_{k-1}\in\qh_{k-1}$ such that $p_k=xp_{k-1}$, therefore
\bean
F_{2}(p_k)&=&F_{2}(xp_{k-1})=p_{k-1}F_{2}(x)+xF_{2}(p_{k-1})=p_{k-1}K_{1}x+xF_{2}(p_{k-1})\\
&=& x\parent{\fracp{K_{1}}{k-1}D(p_{k-1})+F_{2}(p_{k-1})}=x(F_{2}+\fracp{K_{1}}{k-1}D)(p_{k-1})\in\langle x^2\rangle.\eean
 Hence $(F_{2}+\fracp{K_{1}}{k-1}D)(p_{k-1})\in\langle x\rangle.$  Vector field 
 $$\F_{2}+\fracp{K_{1}}{k-1}\D=
 \frac{1}{k-1}\left(
 \begin{array}{c}xk(-qx+(q+r)y)\\ y((p+r)k-p-q-r)x+(-pk+p+q+r)y)\end{array}
 \right)$$ is irreducible if, and only if, $p+q+r\ne pk.$ Applying Lemma \ref{irred1} we have that
$p_{k-1}\in\langle x\rangle$ and consequently $p_k\in\langle x^2\rangle$.

Consider now the case $m=3$. If $F_{2}(p_{k})\in\langle x^3\rangle$ then $F_{2}(p_{k})\in\langle x^2\rangle$ and by the previous paragraph we have that there exists $p_{k-2}\in\qh_{k-2}$ such that $p_k=x^2p_{k-2}$, therefore
\bean F_{2}(p_k)&=&F_{2}(x^2p_{k-2})=p_{k-2}F_{2}(x^2)+x^2F_{2}(p_{k-2})=2p_{k-2}K_{1}x^2+x^2F_{2}(p_{k-2})\\
&=&x^2\parent{\fracp{2K_{1}}{k-2}D(p_{k-2})+F_{2}(p_{k-2})}=x^2(F_{2}+\fracp{2K_{1}}{k-2}D)(p_{k-2})\in\langle x^3\rangle.\eean
Hence $(F_{2}+\fracp{2K_{1}}{k-2}D)(p_{k-2})\in\langle x\rangle$ and as $\F_{2}+\fracp{2K_{1}}{k-2}\D$ is irreducible if, and only if, $p+q+r\ne p\fracp{k}{2},$ 
applying Lemma \ref{irred1} we have that
$p_{k-2}\in\langle x\rangle$ and consequently $p_{k}\in\langle x^3\rangle$.
Reasoning by induction we get the result for $m\in\mathbb{N}$.\\
Reasoning as before, it is easy to prove that for $f_2=y$ and $f_3=x-y,$ the conditions are $p+q+r\ne q\frac{k}{j}$ and  $p+q+r\ne r\frac{k}{j},\ j=1,\dots,\ m-1,$ respectively. \end{proof}

Next statement establishes a cyclicity relation between the co-ranges of the operators $\ell_k$.
\begin{lemma}\label{ciclicidadl}
Consider $\F_2=(x(-qx+(q+r)y),y((p+r)x-py))^T$ with $p,q,r$ natural numbers and $M=p+q+r.$  For $k\ge 2$, it is always possible to choose $\Cor(\ell_{k+M})$, a complementary subspace to $\Range(\ell_{k+M})$, such that $\Cor(\ell_{k+M}) = I_M \Cor(\ell_{k})$ being $I_M=x^py^q(x-y)^r.$
\end{lemma}

\begin{proof}
We first see that both subspaces have the same dimension. Indeed, by  Lemma \ref{irredfm}, $Ker(\ell_k)= \langle I_M^l \rangle$ if $k-1=lM$. Otherwise, $Ker(\ell_k)=\{0\}.$ Thus,
$\dim(\Cor(\ell_{k}))=2$ if $k=lM$ and $\dim(\Cor(\ell_{k}))=1,$ otherwise; i.e.,
$\dim(\Cor(\ell_k))=\dim(\Cor(\ell_{k+M})).$\\

For completing the proof it is enough to prove that $I_M \Cor(\ell_k) \subset \Cor(\ell_{k+M})$ or
equivalently that $I_M\Cor(\ell_{k})\cap
\Range(\ell_{k+M})=\{0\}$ by \emph{reductio ad absurdum}.
 Let $p_k\in\Cor\parent{\Opl_k}\setminus\llave{0}$ such that $p_kI_M\in\Range\parent{\Opl_{k+M}}$, then there exists $p_{k+M-1}\in\qh_{k+M-1}^\t\setminus\llave{0}$ 
 such that $\Opl_{k+M}(p_{k+M-1})=p_kI_M$, that is, $\Opl_{k+M}(p_{k+M-1})$ is multiple of $I_M$. As $\frac{p(k+M-1)}{j}>\frac{pM}{j}>M,\ j=1, \dots, p-1$; $\frac{q(k+M-1)}{j}>M,\ j=1, \dots, q-1$; $\frac{r(k+M-1)}{j}>M,\ j=1, \dots, r-1,$   by applying Lemma \ref{irredfm} we have that $p_{k+M-1}=p_{k-1}I_M$ with $p_{k-1}\in\qh_{k-1}^\t\setminus\llave{0}$ and consequently
\bean p_kI_M&=&F_2(p_{k+M-1})=F_2(p_{k-1}I_M)=I_MF_2(p_{k-1}).\eean
Hence $p_k=F_2(p_{k-1})$, that is, $p_k\in\Range\parent{\Opl_k}\cap\Cor\parent{\Opl_k}$ which gives a contradiction.
\end{proof}

Next result provides an orbital normal form of vector field
whose first homogeneous component is integrable and quadratic Lotka-Volterra type.
This normal form depends on the first integral of the first homogeneous component of the vector field and it is a suitable normal form for the applications.

\begin{theorem}\label{NFhfsimples} 
Vector field $\F=\F_2+\mbox{h.o.t.}$ with $\F_2=(x(-qx+(q+r)y),y((p+r)x-py))^T,\ p,q,r$ natural numbers and $M=p+q+r$ is orbitally equivalent to 
$$\dot{\x}=\F_2+\sum_{j=2}^{M+1}\eta_j^{(0)}\D+\sum_{i=1}^{\infty}\sum_{j=2}^{M+1}\eta_j^{(i)}I_M^i\D,$$
with $\eta_j^{(i)}\in\Cor(\ell_j)$ and $I_M=x^py^q(x-y)^r.$
\end{theorem}

\begin{proof}
Applying Theorem \ref{teoFNlequivFn} and Lemma \ref{lebasepkt}, we can assert that $\F$ is orbital equivalent to $\F_2+\sum_{j\ge 2}\eta_j\D$ with $\eta_j\in\Cor\parent{\Opl_j}$.
In order to finish the proof it is sufficient to apply Lemma \ref{ciclicidadl} for the components of the normal form of degree greater than $M+1$.
\end{proof}

\section{Main results}\label{sec:main}
Our purpose is to characterize the analytic integrable vector fields which are perturbations of quadratic Lotka-Volterra type. For that, we will assume that the lowest degree component of the vector field satisfies the necessary condition of analytic integrability given in Proposition \ref{pronilint}, i.e. we deal with the vector field $\F=\F_2+\mbox{h.o.t.}$ with $\F_2=(x(-qx+(q+r)y),y((p+r)x-py))^T,\ p,q,r\in\Natural,\ \mbox{gcd}(p,q,r)=1.$ 

We give the main result of our study. It solves the analytic integrability problem for vector fields which are perturbations of quadratic Lotka-Volterra  vector fields whose first component is polynomially integrable. It also gives the expression of a first integral.
\begin{theorem}\label{main}
Let $\F=\F_2+\mbox{h.o.t.}$ be with $\F_2=(x(-qx+(q+r)y),y((p+r)x-py))^T,\ p,q,r\in\Natural,\ \mbox{gcd}(p,q,r)=1.$
The vector field $\F$ is analytically integrable if, and only if, it is orbitally equivalent to $\F_2.$\\
Moreover, in such a case, $\F$ has an analytic first integral of the form $I=I_M+\mbox{h.o.t.}$ being $I_M=x^py^q(x-y)^r$ a primitive first integral of $\F_2.$ 
\end{theorem}
\begin{proof}
We see the sufficiency. The polynomial is $I_M=x^py^q(y-x)^r$ is a first integral of $\F_{2}$ which it is transformed into a formal first integral $I=I_M+\mbox{h.o.t.}$ of $\F$ and from [Theorem A,\cite{Mattei80}] $\F$ is analytically integrable. \\
We see the necessity of the condition. 
Applying Theorem \ref{teoFNlequivFn} and Lemma \ref{lebasepkt}, we can assert that $\F$ is orbital equivalent to $\G=\F_2+\sum_{j\ge 2}\eta_j\D$ with $\eta_j\in\Cor\parent{\Opl_j}$. \\
Let note that $\F$ has an analytic first integral equivalents to $\G$ has a formal first integral. Assume that $\G$ is formally integrable and not all the $\eta_j$ are zero. Let $N$ defined by
$N=\min\llave{j>1 : \eta_j\not\equiv 0}$. A formal first integral of $\G$ is of the form  $I=I_M^l+\sum_{j>Ml}I_j$ with $I_j\in\qh_j$. Imposing the integrability condition we have
\bean 0&=&(G(I))_{N+Ml}=(\eta_{N}D)(I_M^l)+F_2(I_{Ml+N-1})\\
&=&Ml\eta_{N}I_M^l+\Opl_{Ml+N}\parent{I_{Ml+N-1}}.\eean
But this equation is incompatible since by Lemma \ref{ciclicidadl} $Ml\eta_NI_M^l\in\Cor\parent{\Opl_{Ml+N}}$ and $\Opl_{Ml+N}\parent{I_{Ml+N-1}}=-Ml\eta_NI_M^l\in\Range\parent{\Opl_{Ml+N}}$ which is a contradiction.
Consequently, $\G=\F_2,$ i.e. $\F$ is orbitally equivalent to $\F_2.$

We now see the second part. First integrals of $\F_2$ are  $\Psi(I_M)$ for any formal function $\Psi$. So, first integrals of $\F$ are $\Psi(I_M+\mbox{h.o.t.})$ since $\F$ is orbitally equivalent to $\F_2$. Thus, $I_M+\mbox{h.o.t.}$ is also a first integral of $\F$. 
\end{proof}

The following theorem characterizes the analytic integrability of a vector field whose first homogeneous component is quadratic Lotka-Volterra type through the existence of a Lie symmetry.

\begin{theorem}\label{main2}
Let $\F=\F_2+\mbox{h.o.t.}$ be with $\F_2=(x(-qx+(q+r)y),y((p+r)x-py))^T,\ p,q,r\in\Natural,\ \mbox{gcd}(p,q,r)=1.$
Then $\F$ is analytically integrable if, and only if, there exist a formal vector field $\G=\sum_{j\geq 1}\G_j$, $\G_j\in\QH_j$, $\G_1=(x,y)^T$ and a formal scalar function $\nu$, $\nu(\zero)=1$ such that $[\F,\G]=\nu\F$, i.e. $\F$ has a Lie symmetry.
\end{theorem}
The proof of Theorem \ref{main2}  follows from Theorem \ref{main} and applying \cite[Theorem 1.3]{Algaba09_like}.\\

We solve the analytic integrability problem through the existence of a formal inverse integrating factor. 

\begin{theorem}\label{main3}
Let $\F=\F_2+\mbox{h.o.t.}$ be with $\F_2=(x(-qx+(q+r)y),y((p+r)x-py))^T,\ p,q,r\in\Natural,\ \mbox{gcd}(p,q,r)=1.$
Then $\F$ is analytically integrable if, and only if, it has a formal inverse integrating factor of the form  $V=xy(x-y)+\mbox{h.o.t.}.$ 
\end{theorem} 
\begin{proof}
We prove that the condition is necessary. We assume that $\F$ is analytically integrable. From Theorem \ref{main}, it is orbitally equivalent to $\F_2=\X_h+\mu\D$ being $h=\frac{p+q+r}{3}xy(x-y)$ and $\mu=\frac{1}{3}((-2q+p+r)x+(q+r-2p)y),$ which has the inverse integrating factor $h.$ Undoing the change, it has that $\F$ has a formal inverse integrating $V=h+\mbox{h.o.t.} .$\\

Now we will see the sufficiency of the condition. Let $V=h+\mbox{h.o.t.}$ a formal inverse integrating factor of $\F$. Since Theorem \ref{teoFNlequivFn} and Lemma \ref{lebasepkt}, we can assert that $\F$ is orbital equivalent to $\G=\F_2+\sum_{j\ge 2}\eta_j\D$ with
$\eta_j\in\Cor\parent{\Opl_j}$. Therefore, $\F$ has a formal inverse integrating factor if, and only if, $\G$ has it too. Moreover, the formal inverse integrating factor $W$ of  $\G$ is also of the form $W=h+\mbox{h.o.t.}.$ On the other hand, the unique invariant curves of $\G$ are $x,y,x-y$ and any $u$ formal with $u(\0)=1,\ u$ is an unit element. So, we get $W=hu$ being $u$ formal and $u(\0)=1.$ Equation  $G(W)-W\mbox{div}(\G)=0$ is
$$0=uG(h)+hG(u)-hu\mbox{div}(\G).$$
As $G(h)=3h\mu+	\sum_{j>2}3h\eta_j$ and $\mbox{div}(\G)=3\mu+\sum_{j>2}(j+2)\eta_j,$ it has that $$0=h(G(u)-u\sum_{j>2}(j-1)\eta_j).$$
Expanding $u=1+\sum_{i\ge 1}u_i,$ it is easy to prove that the equation to degree $i+1$ becomes
\begin{equation}\label{equfii} 0=G_2(u_i)-i\eta_{i+1}+\sum_{k=1}^{i-1}(2k-i)\eta_{i-k+1} u_k
\end{equation}
We see that  $\eta_j=0$ for all $j$. Indeed, otherwise, let $j_0=\min\llave{j\in\Natural : \eta_{j+1}\not\equiv 0}$. 
Equation (\ref{equfii}) to degree $j_0+1$ is 
$$G_2(u_{j_0})=j_0\eta_{j_0+1}-\sum_{k=1}^{j_0-1}(2k-j_0)\eta_{j_0-k+1} u_k.$$
As $\eta_{j_0-k+1}= 0$ for $1\le k\le j_0-1$, we get $G_2(u_{j_0})=j_0\eta_{j_0+1},$ i.e.
$\eta_{j_0+1}\in\Cor\parent{\Opl_{j_0+1}}$ and $\eta_{j_0+1}\in\Range\parent{\Opl_{j_0+1}}.$ We conclude that $\eta_{j_0+1}=0.$
\end{proof}


\section{An application}\label{sec:aplica}
Consider the analytic integrability problem of the following system
\begin{equation}\label{sis:aplica}\left(\begin{array}{c}\dot{x}\\ \dot{y}\end{array}\right)=\left(\begin{array}{c}x(3y-x)\\ y(3x-y)\end{array}\right)+\left(\begin{array}{c}x(a_{20}x^2+a_{11}xy+a_{02}y^2)\\ y(b_{20}x^2+b_{11}xy+b_{02}y^2)\end{array}\right).\end{equation}

The first homogeneous component of the vector field is $\F_2=(x(-x+3y),\ y(3x-y))^T,$ where $\F_2=\X_h+\mu \D$ with $h=\frac{4}{3}xy(x-y)$ and $\mu=\frac{1}{3}(x+y).$ The vector field $\F_2$ is polynomially integrable and a primitive first integral is $I_4=xy(x-y)^2.$

The following result solves the integrability problem for this family.
\begin{theorem}\label{thm:aplica} System (\ref{sis:aplica}) is analytically integrable if, and only if, 
one of the following conditions holds:\\
\noindent (1) $b_{11}+5b_{02}=b_{20}+2b_{02}=a_{11}+3b_{02}=a_{20}-b_{02}=a_{02}=0,$\\
\noindent (2) $b_{11}+3b_{02}=a_{02}+2b_{02}=a_{11}+5b_{02}=a_{20}-b_{02}=b_{20}=0,$\\
\noindent (3) $2a_{11}+2a_{02}-3b_{20}-3b_{11}-5b_{02}=a_{02}b_{20}+a_{02}b_{11}+3a_{02}b_{02}+2b_{20}b_{02}=2a_{20}+b_{20}+b_{11}+3b_{02}=0,$\\ 
\noindent (4)  $a_{02}+5b_{02}=a_{11}+b_{11}=5a_{20}+b_{20},$\\
\noindent (5) $a_{11}+b_{11}=a_{20}+b_{02}=a_{02}=b_{20},$\\
\noindent (6) $b_{20}-b_{02}=a_{02}+b_{02}=a_{11}+b_{11}=a_{20}+b_{02}.$
\end{theorem}
\begin{proof} To prove the necessary condition, it has computed the first coefficients of the normal form given in Theorem \ref{NFhfsimples}. By Theorem \ref{main}, the vanishing of the coefficients leads us to the integrability.  In this case, it has been necessary the coefficients of the normal form up order 7, 
$$(\dot{x},\dot{y})^T=\F_2+(\alpha_2x^2+\alpha_3x^3+\alpha_4x^4+\alpha_5xI_4+\beta_5yI_4+
\alpha_6x^2I_4)(x,y)^T.$$
The coefficients  $\alpha_2,\alpha_3,\alpha_4,\alpha_5,\alpha_6$ and $\beta_5$ are polynomials too long, so we do not given them here. Their vanishing arrives to systems (\ref{sis:aplica})  for cases 1--6.

We prove the sufficiency. 
\noindent System (\ref{sis:aplica}) for case 1 has an analytic first integral 
$xy(x-y-b_{02}x^2+\frac{1}{3}b_{02}xy)^2(1-b_{02}x-b_{02}y)^{-3}.$\\
\noindent System (\ref{sis:aplica}) for case 2 is transformed into system (\ref{sis:aplica}) for case 1 by using the involution $(x,y)\leftrightarrow(y,x).$\\
\noindent System (\ref{sis:aplica}) for case 3 has an inverse integrating factor 
$xy(x-y)(2+b_{20}x+b_{11}x+3b_{02}x-2b_{02}y)$ whose first component is $h$. Applying Theorem \ref{main3}, the vector field is analytically integrable. \\
\noindent System (\ref{sis:aplica}) for case 4 has a polynomial first integral
$xy(3x-3y-3a_{20}x^2+b_{11}xy+3b_{02}y^2)^2.$\\
\noindent System (\ref{sis:aplica}) for case 5 has an analytic first integral
$$xy(3x-3y+(b_{02}+b_{11})xy)^2(2+2b_{02}x-2b_{02}y+(b_{02}b_{11}+b_{02}^2)xy)^{-3}.$$
\noindent System (\ref{sis:aplica}) in the case 6, for $b_{11}=-2b_{02},$ has an inverse integrating factor
$xy(x-y)(1+b_{02}x-b_{02}y).$ Otherwise, we have not found the expression of an inverse integrating factor starting by $h$. In this case, we center on proving its existence in order to apply Theorem \ref{main3}.\\
Consider $V=xyC(x,y)$ with $C$ the invariant curve given by Lemma \ref{lem:case6}. It has that 
$F(V)=xyF(C)+xCF(y)+yCF(x)=V(K^{(1)}+K^{(2)}+K^{(3)}) $ with $K^{(1)},K^{(2)}$ and $K^{(3)}$ the cofactors of $x,y$ and $C$, respectively, $K^{(1)}=-x+3y-b_{02}x^2-b_{11}xy-b_{02}y^2,\ K^{(2)}=3x-y+b_{02}x^2+b_{11}xy+b_{02}y^2$ and $K^{(3)}=-(x+y)(1+2b_{02}x-2b_{02}y)$ and as $K^{(1)}+K^{(2)}+K^{(3)}=x+y-2b_{02}x^2+2b_{02}y^2=\mbox{div}(\F),$  $V$ is an inverse integrating factor of $\F$. This concludes the proof. 
\end{proof}
\begin{lemma}\label{lem:case6}System (\ref{sis:aplica}) for case 6 has an invariant curve $C=x-y+\mbox{h.o.t.}$ with cofactor $K=-(x+y)(1+2b_{02}x-2b_{02}y).$
\end{lemma}
\begin{proof} System (\ref{sis:aplica}) for case 6 is $\dot{\x}=\F_2+\F_3$ with $\F_2=(x(-x+3y),\ y(3x-y))^T$ and 
$\F_3=(x(-b_{02}x^2-b_{11}xy-b_{02}y^2),\ y(b_{02}x^2+b_{11}xy+b_{02}y^2))^T.$\\
We claim that there exists a formal invariant curve of $\F$ of the form $C=\sum_{j\ge1}C_j$ with 
\begin{equation}\label{expC}C_{2j-1}=A_{2j-1}x^{j-1}y^{j-1}(x-y),\quad C_{2j}=x^{j-1}y^{j-1}(A_{2j}x^2+B_{2j}xy-A_{2j}y^2),\end{equation}
for any $j\ge 1,$ with cofactor $K_1+K_2$ being $K_1=-x-y$ and $K_2=-2b_{02}x^2+2b_{02}y^2.$
We are going to verify that $C$ satisfies $F(C)-KC=0$ degree to degree.\\
For the degree 2, $F_2(C_1)-K_1C_1=0$ arrives to $C_1=x-y$, and for the degree 3, $F_2(C_2)+F_3(C_1)-C_2K_1-C_1K_2=0,$ we get $C_2=b_{02}x^2+\frac{1}{3}(b_{11}-4b_{02})xy+b_{02}y^2.$ Thus, $C_1$ and $C_2$ have the form given by (\ref{expC}). Assume that  (\ref{expC}) is true for $2j_0-1$ and $2j_0$ and we prove that also it holds for $2j_{0}+1$ and $2j_0+2.$\\
Equation $F(C)-KC=0$ for degree $2j_0+2$ is $$F_2(C_{2j_0+1})-C_{2j_0+1}K_1=-F_3(C_{2j_0})+C_{2j_0}K_2=2x^{j_0}y^{j_0}(x-y)(x+y)(A_{2j_0}b_{11}-B_{2j_0}b_{02}).$$ A solution of this equation is $ C_{2j_0+1}=(A_{2j_0}b_{11}-B_{2j_0}b_{02})x^{j_0-1}y^{j_0-1}(x-y),$ i.e. $ C_{2j_0+1}$ is of the form given by (\ref{expC}) with $A_{2j_0+1}=A_{2j_0}b_{11}-B_{2j_0}b_{02}.$\\ 
Analogously, equation $F(C)-KC=0$ for degree $2j_0+3$ is 
$$F_2(C_{2j_0+2})-C_{2j_0+2}K_1=-A_{2j_0+1}x^{j_0}y^{j_0}(x+y)(b_{02}x^2-(b_{11}+4b_{02})xy+b_{02}y^2).$$ A solution of this equation is 
$$ C_{2j_0+2}=A_{2j_0+1}x^{j_0}y^{j_0}\left(-\frac{b_{02}}{2j_0-1}x^2+(\frac{b_{11}}{2j_0+3}+\frac{4(2j_0+1)b_{02}}{(2j_0-1)(2j_0+3)})xy+\frac{b_{02}}{2j_0-1}y^2\right),$$ i.e. $ C_{2j_0+2}$ is of the form given by (\ref{expC}). Therefore, the result is proved. 
\end{proof}

\vspace{0.5truecm}

\noindent {\bf Acknowledgments.} The authors are partially supported by a MINECO/FEDER
grant number MTM2014-56272-C2-2 and by the \emph{Consejer\'{\i}a de Educaci\'on y
Ciencia de la Junta de Andaluc\'{\i}a} (projects P12-FQM-1658, FQM-276).


\begin{thebibliography}{10}

\bibitem{Algaba_Checa_Garcia_Gine2015}
{\sc A. Algaba, I. Checa, C. Garc\'{\i}a, J. Gin\'e}, {\it Analytic integrability inside a family of degenerate centers}, Nonlinear Anal. Real World Appl. {\bf 31} (2016), 288--307.

\bibitem{AlgabaNonlinearity09}
{\sc A. Algaba,  E. Gamero, C. Garc\'{\i}a}, {\it The integrability problem
for a class of planar systems}, Nonlinearity, \textbf{22} (2009), 395--420.


\bibitem{ACG0}
{\sc A. Algaba, C. Garc\'{\i}a, J. Gin\'e}, {\it Analytic integrability for some degenerate planar systems}, Commun. Pure Appl. Anal. \textbf{12} (2013), no. 6, 2797--2809.

\bibitem{ACG}
{\sc A. Algaba, C. Garc\'{\i}a, J. Gin\'e}, {\it Analytic integrability for some degenerate planar vector fields}, J. Differential Equations \textbf{257} (2014), no. 2, 549--565.

\bibitem{AGG17}
{\sc A. Algaba, C. Garc\'{\i}a, J. Gin\'e}, {\it The analytic integrability problem for perturbation of non-hamiltonian quasi-homogeneous nilpotent systems}, Submitted for publication.



\bibitem{Algaba09_like}
{\sc A. Algaba, C. Garc\'{\i}a, M. Reyes}, {\it Like-linearizations of vector fields}, Bull. Sci. Math. \textbf{133} (2009), 806--816.








\bibitem{CGRS2012} {\sc X. Chen, J. Gin\'e, V.G. Romanovski, D.S. Shafer}, {\it The 1:-q resonant center problem for certain cubic Lotka-Volterra systems}, Applied Mathematics and Computation, {\bf 218} (2012), 11620--11633.


\bibitem{CMR} {\sc C. Christopher, P. Marde\v{s}ic, C. Rousseau}, {\it Normalizable, integrable, and linearizable saddle points for complex quadratic systems in $\mathbb{C}^2$}, J. Dynam. Control Systems {\bf 9} (2003), no. 3, 311--363.

\bibitem{CR} {\sc C. Christopher, C. Rousseau}, {\it Normalizable, integrable and linearizable saddle points in the Lotka-Volterra system}, Qual. Theory Din. Syst.  {\bf 5} (2004), 1, 11-61.


\bibitem{DGOR2013} {\sc D. Doli\'canin, J. Gin\'e, R. Oliveira, V.G. Romanovski}, {\it The center problem for a 2:-3 resonant cubic Lotka-Volterra system}, Applied Mathematics and Computation, {\bf 220} (2013), 12--19.











\bibitem{GR} {\sc J. Gin\'e, V.G. Romanovski}, {\it Integrability conditions for Lotka-Volterra planar complex quintic systems}, Nonlinear Anal. Real World Appl. {\bf 11} (2010), no. 3, 2100--2105.

\bibitem{HJ} {\sc M. Han, K. Jiang}, {\it Normal forms of integrable systems at a resonant saddle},
Ann. Differential Equations {\bf 14} (1998), no. 2, 150--155.

\bibitem{LCC} {\sc C. Liu, G. Chen, G. Chen}, {\it Integrability of Lotka-Volterra type systems of degree 4},
J. Math. Anal. Appl. {\bf 388} (2012), no. 2, 1107--1116.

\bibitem{LCL} {\sc C. Liu, G. Chen, C. Li}, {\it Integrability and linearizability of the Lotka-Volterra systems},
J. Differential Equations {\bf 198} (2004), no. 2, 301--320.


\bibitem{Mattei80} {\sc J. F. Mattei, R. Moussu}, {\it Holonomie et int\'egrales premi\`eres},
Ann. Sci. \'Ecole Norm. Sup. (4), {\bf 13} (1980), no. 4, 469--523.


\bibitem{Poincare2} {\sc H. Poincar\'e}, {\it M\'emoire sur les courbes d\'efinies par les \'equations diff\'erentielles}, J Math Pures Appl {\bf 4} (1885) 167--244; Oeuvres de Henri Poincar\'e, vol. I. Paris: Gauthier-Villars; 1951. pp. 95--114




\bibitem{WH2014}
{\sc Q. Wang, W. Huang}, {\it Integrability and linearizability for Lotka-Volterra systems with the 3:-q resonant saddle point},
Advances in Difference Equations, (2014), 2014:23.




\end{thebibliography}
\end{document}